\documentclass[11pt]{article}%
\usepackage{amssymb}
\usepackage{eurosym}
\usepackage{amsfonts}
\usepackage{amsmath}
\usepackage{graphicx}
\usepackage{epstopdf}%
\providecommand{\U}[1]{\protect\rule{.1in}{.1in}}
\newtheorem{theorem}{Theorem}
\newtheorem{acknowledgement}[theorem]{Acknowledgement}

\newtheorem{definition}[theorem]{Definition}

\newtheorem{lemma}[theorem]{Lemma}

\newtheorem{proposition}[theorem]{Proposition}
\newtheorem{remark}[theorem]{Remark}

\newenvironment{proof}[1][Proof]{\noindent\textbf{#1.} }{\ \rule{0.5em}{0.5em}}
\begin{document}

\title{Critical graph of a polynomial quadratic differential related to a
Schr\"{o}dinger equation with quartic potential}
\author{Mondher Chouikhi and Faouzi Thabet\\University of Gab\`{e}s, Tunisia}
\maketitle

\begin{abstract}
In this paper, we study the weak asymptotic in the $%
%TCIMACRO{\U{2102} }%
%BeginExpansion
\mathbb{C}
%EndExpansion
$ plane of some wave functions resulting from the WKB techniques applied to a
Shrodinger equation with quartic oscillator and having some boundary
condition. In first step, we make transformations of our problem to obtain a
Heun equation satisfied by the polynomial part of the WKB wave functions
.Especially , we investigate the properties of the Cauchy transform of the
root counting measure of a re-scaled solutions of the Schrodinger equation, to
obtain a quadratic algebraic equation of the form $\mathcal{C}^{2}\left(
z\right)  +r\left(  z\right)  \mathcal{C}\left(  z\right)  +s\left(  z\right)
=0$, where $r,s$ are also polynomials. In second step, we discuss the
existence of solutions (as Cauchy transform of a signed measures) of this
algebraic equation.This problem remains to describe the critical graph of a
related 4-degree polynomial quadratic differential $-p\left(  z\right)
dz^{2}$. In particular, we discuss the existence(and their number) of finite
critical trajectories of this quadratic differential.

\end{abstract}



\section{\bigskip Introduction}

Quadratic differentials appear in different mathematical and
mathematical-physical domains, such as, orthogonal polynomials, potential
theory, ordinary differential equations, quantum theory, moduli spaces ...

Recently, quadratic differentials have provided an important tool in the
asymptotic study of some solutions of algebraic equations.

In quantum theory, trajectories of some quadratic differentials have crucial
role in WKB analysis, more precisely, consider the time-independent
Schr\"{o}dinger equation on the complex plane $%
%TCIMACRO{\U{2102} }%
%BeginExpansion
\mathbb{C}
%EndExpansion
$
\begin{equation}
-\frac{\text{$\hslash$}^{2}}{2m}\frac{d^{2}\psi(z,\text{$\hslash$})}{dz^{2}%
}+V(z,\text{$\hslash$})\psi(x,\text{$\hslash$})=E\psi(z,\text{$\hslash$}).
\label{eq1}%
\end{equation}
where $\hslash$ denotes the constant of Planck, $z,$ the local coordinate,
$\psi$, the wave function,$V\left(  z,\hslash\right)  ,$ the potential, and
$E$, the energy of a system of mass $m.$ The determination of an explicit
solution of equation \ref{eq1} is difficult in general case but a general
solution can be written as a linear combination of particular solutions.The
series expansion in $\hslash^{-1}$ of $\hslash$$^{-2}V\mathcal{(}z,\hslash$)
provide a principal part $V_{0}$ which defines a meromorphic quadratic
differential $\phi$ in the complex plane . The critical graph of $\phi$
\textbf{(}the closure of the union of critical trajectories of $\phi$) ,
called also \textit{Stokes graph,} is important for the WKB analysis which
gives an important method to determine particular solution of Schr\"{o}dinger
equation. General study is more detailed in \cite{ref1} and \cite{ref2}, where
they considered the "classical" case $\hslash$$^{-2}V\left(  z,\hslash\right)
=V_{0}(z).$ The classical Schr\"{o}dinger equation is
\[
H\psi(z)=E\psi(z)
\]
where
\[
H=-\frac{\text{$\hslash$}^{2}}{2m}\frac{d^{2}}{dz^{2}}+V(z)
\]
is the hamiltonian. It is obvious that $\psi$ is an eigenfunction of the
operator $H$. In quantum mechanics, $H$ is generally supposed to be hermitian;
in fact this condition guarantees a crucial property of quantum theory : the
fact that the energy levels are real. Clearly these energy levels are the
eigenvalues of $H.$It is shown that the Hermeticity condition of $H$ is
sufficient and not necessary; see \cite{ref3}. In fact, under PT-symmetry
condition ( space-time reflection ) Bender has shown in \cite{ref4} that the
hamiltonian $H=-\frac{d^{2}}{dx^{2}}+ix^{3}$ (which is not hermitian) has a
real spectre. PT-symmetry condition allows to study many of hamiltonian that
were discarded before.

The PT-symmetry condition was at the origin of many works, especially the
numerical and asymptotic studies of the spectrum of operators $H,$ when $V$ is
polynomial. Bender and Boettcher conjectured in \cite{ref7} that when
$V(z)=(iz)^{m}+\alpha z^{2}$ $(\alpha\in%
%TCIMACRO{\U{211d} }%
%BeginExpansion
\mathbb{R}
%EndExpansion
),$ the eigenvalues of $H$ are all real and positive. Many works supported
this conjecture, \cite{ref8}, \cite{ref9}, etc...

The first rigourous proof of reality and positivity of the eigenvalues of some
non-self-adjoint hamiltonian $H$ was given by Dorey, Dunning and Tateo in
\cite{ref10}. However, there are some PT-symmetric hamiltonian with polynomial
potentials producing non real eigenvalues, see \cite{ref11}. This PT-symmetric
condition of the hamiltonian with polynomial potentials can be expressed by
\begin{equation}
\overline{V(-\overline{z})}=V(z),z\in%
%TCIMACRO{\U{2102} }%
%BeginExpansion
\mathbb{C}
%EndExpansion
. \label{*}%
\end{equation}

In this paper, we focus on a case of quartic ( $V$ is a polynomial of degree
$4$) PT-symmetric hamiltonian and we study the weak asymptotic of its
spectrum. Starting from a Schr\"{o}dinger differential equation and using the
Cauchy transform defined in (\ref{cauchy}) we get algebraic equation
(\ref{algeqtiion}), which gives rise to a particular quadratic differential.%
\[
\text{Schr\"{o}dinger equation}\Rightarrow\text{ Algebraic equation}%
\Rightarrow\text{Auadratic differential.}%
\]

More precisely, we consider the eigenvalue problem%

\begin{equation}
\left\{
\begin{array}
[c]{c}%
-y^{\prime\prime}+(t^{4}+2bm^{\frac{2}{3}}t^{2}+2mit)y=\lambda y,\\
\\
y(te^{-i(\frac{\pi}{2}\pm\frac{\pi}{3})})\longrightarrow0,t\longrightarrow
\infty,
\end{array}
\right.  \label{system}%
\end{equation}
where $b$ is real, $m$ is integer and we choose $\hslash=1;$ see
\cite{ref3},\cite{ref5},\cite{ref6}. This problem is a quasi-exactly solvable
: for each $b,$ there are $m$ eigenvalues $\lambda_{m,k}$, with elementary
eigenfunctions
\[
y_{m,k}(t)=p_{m,k}(t)e^{(-i\frac{t^{3}}{3}-ibm^{\frac{2}{3}}t)},
\]
and $p_{m,k}$ are polynomials of degree at most $m-1;$ $k=1,...,m$.

Under condition (\ref{*}), the potential
\[
V(t)=t^{4}+2bm^{\frac{2}{3}}t^{2}+2mit
\]
is PT-symmetric. With the rotation $z=it,$ we obtain the following system%
\begin{equation}
\left\{
\begin{array}
[c]{c}%
-y^{\prime\prime}+(z^{4}-2bm^{\frac{2}{3}}z^{2}+2mz)y=\lambda y,\\
\\
y(te^{\pm i\frac{\pi}{3})})\longrightarrow0,t\rightarrow\infty\text{,}%
\end{array}
\right.  \label{system2}%
\end{equation}
\textit{\ } with $y\in\left\{  pe^{\tau};\deg p\leq m-1\right\}  $ and
$\tau(z)=\frac{z^{3}}{3}-bm^{\frac{2}{3}}z.$

For the sick of simplicity, we denote $y_{m,m}$ by $y_{m}$ and $p_{m,m}$ by
$p_{m}$. Substituting in (\ref{system2}), we obtain
\begin{equation}
p_{m}^{\prime\prime}-2\tau_{m}^{^{\prime}}p_{m}^{^{\prime}}+[2(m-1)z-(\lambda
_{m}+b^{2}m^{\frac{4}{3}})]p_{m}=0. \label{**}%
\end{equation}
It's clear that $p_{m}$ are also the eigenfunctions of the operator
\[
T_{m}(y)=y^{\prime\prime}-2\tau^{^{\prime}}y^{^{\prime}}+2(m-1)zy,
\]
associated to the eigenvalue $\beta_{m}=\lambda_{m}+b^{2}m^{\frac{4}{3}}.$

Let us denote by $\mathcal{C}_{\nu_{m}}$ the cauchy transform of $\nu_{m}%
=\nu(p_{m})$ where $p_{m}$ is the polynomial part of the eigenfunction
$y_{m}.$ The eigenvalue problem (\ref{system}) has infinitely eigenvalues
$\lambda_{m}$ tending to infinity, so, in order to study the asymptotic of the
sequence $\nu_{m}=\nu(p_{m}),$ we shall consider $q_{m}(z)=p_{m}%
(m^{\varepsilon}z)$ a re-scaled polynomial of $p_{m}$ and
\[
\mathcal{\rho}_{m}(z)=\frac{q_{m}^{^{\prime}}(z)}{mq_{m}(z)}=m^{\varepsilon
}\mathcal{C}_{\nu_{m}}(m^{\varepsilon}z)
\]
Substituting in (\ref{**}), we get that%
\[
\mathcal{C}_{\nu_{m}}^{2}(z)-2\frac{(z^{2}-bm^{\frac{2}{3}})}{m}%
\mathcal{C}_{\nu_{m}}(z)+(\frac{2z}{m}-\frac{\beta_{m}}{m^{2}})+\frac
{\mathcal{C}_{\nu_{m}}^{^{\prime}}(z)}{m}=0,
\]
and%
\[
\mathcal{\rho}_{m}^{2}(z)-2(m^{3\varepsilon-1}z^{2}-m^{\varepsilon-\frac{1}%
{3}}b)\mathcal{\rho}_{m}(z)+(2m^{3\varepsilon-1}z-m^{2\varepsilon-2}\beta
_{m})+\frac{\mathcal{\rho}_{m}^{^{\prime}}(z)}{m}=0.
\]
In particular, for $\varepsilon=\frac{1}{3},$ we get%
\[
\mathcal{\rho}_{m}^{2}(z)-2(z^{2}-b)\mathcal{\rho}_{m}(z)+(2z-\beta
_{m}/m^{\frac{4}{3}})+\frac{\mathcal{\rho}_{m}^{^{\prime}}(z)}{m}=0
\]

It was shown in \cite{ref5}, that $\left(  \beta_{m}/m^{\frac{4}{3}}\right)  $
is bounded. By the Helly selection theorem, we may assume that
\[
\underset{m\rightarrow\infty}{\lim}\left(  \beta_{m}/m^{\frac{4}{3}}\right)
=\beta,
\]
and then, there exists a compactly supported positive measure $\nu$ such that
\[
\underset{m\rightarrow\infty}{\lim}\nu_{m}=\nu,\underset{m\rightarrow
\infty}{\lim}\mathcal{\rho}_{m}=\mathcal{C}.
\]
Finally, we obtain the algebraic equation :
\begin{equation}
\mathcal{C}^{2}-2(z^{2}-b)\mathcal{C+}(2z-\beta)=0, \label{algeqtiion}%
\end{equation}
where we are looking for those solutions as Cauchy transform of some
compactly-supported Borelian positive measure in the complex plane $%
%TCIMACRO{\U{2102} }%
%BeginExpansion
\mathbb{C}
%EndExpansion
$; any connected curve of the support of such a measure (if exists) coincides
with a finite critical trajectory of a quadratic differential on the Riemann
sphere $\widehat{%
%TCIMACRO{\U{2102} }%
%BeginExpansion
\mathbb{C}
%EndExpansion
}$ of the form $-p\left(  t\right)  dt^{2},$ where $p$ is polynomial of degree
$4,$(Subsection \ref{eqal}). The general possible configurations of the
critical graph of this kind of quadratic differential are will known (see ,,),
but the main goal of this work (besides the precise description of the
critical graph) is the investigation of the number of finite critical
trajectories when $p$ is a real polynomial, (Subsection \ref{qd}).

We notice that if we start with the eigenvalue problem%
\[
\left\{
\begin{array}
[c]{c}%
-y^{\prime\prime}+(t^{4}+2b_{m}t^{2}+2mit)y=\lambda y\\
\\
y(te^{-i(\frac{\pi}{2}\pm\frac{\pi}{3})})\rightarrow0,t\rightarrow\infty,
\end{array}
\right.
\]
\ then we analyse the case when $b_{m}=\mathcal{O(}m^{\frac{2}{3}}).$The case
$\underset{m\rightarrow\infty}{\lim}b_{m}/m^{\frac{2}{3}}=0$ is studied in
\cite{ref6}.

\section{Quadratic differentials}

\bigskip Below, we describe the critical graphs of the family of quadratic
differentials on the Riemann sphere $\overline{%
%TCIMACRO{\U{2102} }%
%BeginExpansion
\mathbb{C}
%EndExpansion
}$%
\[
\varpi\left(  z\right)  =-p\left(  z\right)  dz^{2},
\]
where $p$ is a monic real polynomial of degree $4.$

We begin our investigation by some immediate observations from the theory of
quadratic differentials. For more details, we refer the reader to
\cite{Strebel},\cite{jenkins}...

Recall that \emph{finite critical points }of the polynomial quadratic
differential $\varpi$ are its zeros and simple poles; poles of order $2$ or
greater then $1$ called \emph{infinite critical points. }All other points of
$\overline{%
%TCIMACRO{\U{2102} }%
%BeginExpansion
\mathbb{C}
%EndExpansion
}$ are called \emph{regular points}.

With the parametrization $u=1/z$, we get
\[
\ \varpi\left(  u\right)  =\left(  -\frac{1}{u^{8}}+\mathcal{O}(\frac{1}%
{u^{7}})\right)  du^{2},\quad u\rightarrow0;
\]
thus, infinity is an \emph{infinite critical point }of $\varpi,$ as a pole of
order $8$ and $6$ respectively. \emph{Horizontal trajectories} (or just
trajectories) of the quadratic differential $\varpi$ are the zero loci of the
equation%
\begin{equation}
\mathcal{\Re}\int^{z}\sqrt{p\left(  t\right)  }\,dt=\text{\emph{const}},
\label{eq traj}%
\end{equation}
or equivalently%
\[
\varpi=-p\left(  z\right)  dz^{2}>0.
\]
If $z\left(  t\right)  ,t\in%
%TCIMACRO{\U{211d} }%
%BeginExpansion
\mathbb{R}
%EndExpansion
$ is a horizontal trajectory, then the function
\[
t\longmapsto\Im\int^{t}\sqrt{p\left(  z\left(  u\right)  \right)  }z^{\prime
}\left(  u\right)  du
\]
is monotone.

The \emph{vertical} (or, \emph{orthogonal}) trajectories are obtained by
replacing $\Im$ by $\Re$ in equation (\ref{eq traj}). The horizontal and
vertical trajectories of the quadratic differential $\varpi$ produce two
pairwise orthogonal foliations of the Riemann sphere $\overline{%
%TCIMACRO{\U{2102} }%
%BeginExpansion
\mathbb{C}
%EndExpansion
}$.

A trajectory passing through a critical point of $\varpi$ is called
\emph{critical trajectory}. In particular, if it starts and ends at a finite
critical point, it is called \emph{finite critical trajectory }or\emph{\ short
trajectory}, otherwise, we call it an \emph{infinite critical trajectory}. A
short trajectory is called \emph{unbrowken }if it does not pass through any
finite critical points except its two endpoints. The closure the set of finite
and infinite critical trajectories, that we denote by $\Gamma_{p},$ is called
the \emph{critical graph}.

The local structure of the trajectories is as follow :

\begin{itemize}
\item At any regular point, horizontal (resp. vertical) trajectories look
locally as simple analytic arcs passing through this point, and through every
regular point of $\varpi$ passes a uniquely determined horizontal (resp.
vertical) trajectory of $\varpi;$ these horizontal and vertical trajectories
are locally orthogonal at this point.

\item From each zero with multiplicity $r$ of $\varpi,$ there emanate $r+2$
critical trajectories spacing under equal angles $2\pi/(r+2)$. 

\item At the pole $\infty$, there are $6$ asymptotic directions (called
\emph{critical directions}) spacing under equal angle $\pi/3$, and a
neighborhood $\mathcal{U}$ of this pole, such that each trajectory entering
$\mathcal{U}$ stays in $\mathcal{U}$ and tends to $\infty$ following one of
the critical directions. These critical directions are ;
\[
D_{k}=\left\{  z\in%
%TCIMACRO{\U{2102} }%
%BeginExpansion
\mathbb{C}
%EndExpansion
:\arg\left(  z\right)  \equiv0\operatorname{mod}\left(  i\frac{\left(
2k+1\right)  \pi}{6}\right)  \right\}  ;k=0,...,5.
\]
See Figure \ref{1}.
\end{itemize}
\begin{figure}[th]
\centering\includegraphics[height=1.8in,width=2.8in]{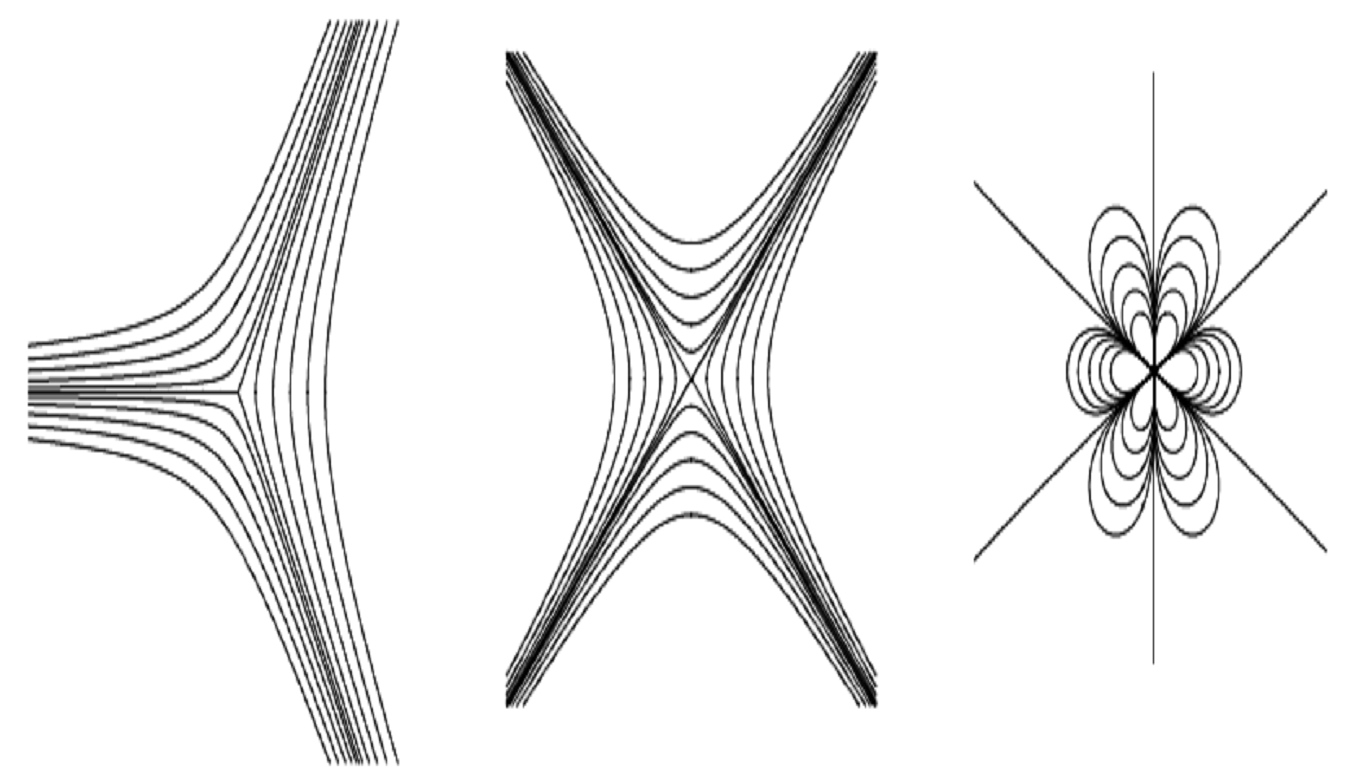}\caption{Structure of the trajectories near a simple zero (left),double zero (center), and pole of order 8 (right)}%
\label{1}%
\end{figure}
We have the following observations :

\begin{itemize}
\item If $\gamma$ is a horizontal trajectory of $\varpi$, then, $\mathcal{\Im
}\int^{z}\sqrt{p\left(  t\right)  }\,dt$\emph{ }is monotone\emph{ }on
$\gamma.$

\item \bigskip If two different trajectories are not disjoint, then their
intersection must be a zero of the quadratic differential.

\item A necessary condition for the existence of a short trajectory connecting
finite critical points of $\varpi$ is the existence of a Jordan arc $\gamma$
connecting them, such that
\begin{equation}
\Re\int_{\gamma}\sqrt{p\left(  t\right)  }dt=0. \label{cond necess}%
\end{equation}
However, we will see that this condition is not sufficient.

\item Since the quadratic differential $\varpi$ has only two poles, Jenkins
Three-pole Theorem (see \cite[Theorem 15.2]{Strebel}) asserts that the
situation of the so-called recurrent trajectory (whose closure might be dense
in some domain in $%
%TCIMACRO{\U{2102} }%
%BeginExpansion
\mathbb{C}
%EndExpansion
$) cannot happen.

\item Any critical trajectory which is not finite diverges to $\infty$
following one of the directions $D_{k}$ described above.
\end{itemize}

\bigskip A very helpful tool that will be used in our investigation is the
Teichm\"{u}ller lemma (see \cite[Theorem 14.1]{Strebel}).

\begin{definition}
\bigskip A domain in $\overline{%
%TCIMACRO{\U{2102} }%
%BeginExpansion
\mathbb{C}
%EndExpansion
}$ bounded only by segments of horizontal and/or vertical trajectories of
$\varpi$ (and their endpoints) is called $\varpi$-polygon.
\end{definition}

\begin{lemma}
[Teichm\H{u}ller]\label{teich lemma} Let $\Omega$ be a $\varpi$-polygon, and
let $z_{j}$ be the critical points on the boundary $\partial\Omega$ of
$\Omega,$ and let $t_{j}$ be the corresponding interior angles with vertices
at $z_{j},$ respectively. Then%
\begin{equation}
\sum\left(  1-\dfrac{\left(  n_{j}+2\right)  t_{j}}{2\pi}\right)  =2+m,
\label{Teich equality}%
\end{equation}
where $n_{j}$ are the multiplicities of $z_{j}=1,$ and $m$ is the number of
zeros of $\varpi$ inside $\Omega.$
\end{lemma}

\subsection{\bigskip Critical graph of $\varpi$ \label{qd}}

We focus on the case when $p$ is a real monic polynomial having two simple
real zeros and two conjugate zeros. By a linear change of variables, we may
assume that
\[
p\left(  z\right)  =\left(  z^{2}-1\right)  \left(  z-a\right)  \left(
z-\overline{a}\right)  ,
\]
where $a\in%
%TCIMACRO{\U{2102} }%
%BeginExpansion
\mathbb{C}
%EndExpansion
_{+}=\left\{  z\in%
%TCIMACRO{\U{2102} }%
%BeginExpansion
\mathbb{C}
%EndExpansion
\mid\Im\left(  z\right)  >0\right\}  .$

\begin{lemma}
\label{residue} For any $a\in%
%TCIMACRO{\U{2102} }%
%BeginExpansion
\mathbb{C}
%EndExpansion
_{+},$ condition (\ref{cond necess}) is satisfied for the pairs of zeros
$\left(  -1,1\right)  $ and $\left(  a,\overline{a}\right)  $.
\end{lemma}

\begin{lemma}
\label{conse teich1}

\begin{itemize}
\item[(i)] The segment $\left[  -1,1\right]  $ is always a short trajectory
connecting $\pm1.$

\item[(ii)] Two critical trajectories emanating from the same zero of
$p\left(  z\right)  $ cannot diverge to $\infty$ in the same direction.

\item[(iii)] No critical trajectory emanating from $z=1$ diverges to $\infty$
in the direction $D_{2}.$

\item[(iv)] Exactly one trajectory emanating from $z=-1$ diverges to $\infty$
in the upper half plane, and it follows the direction $D_{2}.$

\item[(v)] There are at most three short trajectories.
\end{itemize}
\end{lemma}

\bigskip Let us consider the set
\begin{equation}
\Gamma=\left\{  z\in%
%TCIMACRO{\U{2102} }%
%BeginExpansion
\mathbb{C}
%EndExpansion
:\Re\left(  \int_{0}^{z}\sqrt{\left(  t-z\right)  \left(  t-\overline
{z}\right)  \left(  t^{2}-1\right)  }dt=0\right)  \right\}  . \label{eq gamma}%
\end{equation}
Then we have the following observations :

\begin{itemize}
\item the choice of the square root does not play any role in the integral
defining (\ref{eq gamma});

\item $\left[  -1,1\right]  \subset\Gamma;$

\item since $p$ is an real polynomial, it can be shown easily that set
$\Gamma$ is symmetric with respect to the real and imaginary axis;

\item direct calculations show that the points $\pm1+2\exp\left(  \pm
\frac{i\pi}{3}\right)  \in\Gamma.$
\end{itemize}

The following

\begin{lemma}
\label{asymptotic gamma} Let $x$ be the unique real defined by :%
\begin{equation}
\sinh\left(  \frac{2}{3\cot x\sin^{3}x}-\frac{\cot x}{\sin x}\right)  =\cot
x,x\in\left]  0,\frac{\pi}{2}\right[  . \label{asympt en l'inf}%
\end{equation}

Then, we have
\begin{align*}
\underset{z\rightarrow\infty,z\in\Gamma\cap%
%TCIMACRO{\U{2102} }%
%BeginExpansion
\mathbb{C}
%EndExpansion
_{+}^{+}}{\lim}\frac{z}{\left\vert z\right\vert }  &  =\exp\left(  ix\right)
;\\
\underset{z\rightarrow1,z\in\Gamma\cap%
%TCIMACRO{\U{2102} }%
%BeginExpansion
\mathbb{C}
%EndExpansion
_{+}^{+}}{\lim}\frac{z-1}{\left\vert z-1\right\vert }  &  =\exp\left(
i\frac{\pi}{3}\right)  ,
\end{align*}
where we denote by $%
%TCIMACRO{\U{2102} }%
%BeginExpansion
\mathbb{C}
%EndExpansion
_{+}^{+}=\left\{  z\in%
%TCIMACRO{\U{2102} }%
%BeginExpansion
\mathbb{C}
%EndExpansion
\mid\Re\left(  z\right)  ,\Im\left(  z\right)  >0\right\}  .$
\end{lemma}

\begin{remark}
Numeric calculus gives the approximation $x\simeq$ $0.898\simeq\frac{2\pi}%
{7}.$
\end{remark}

\begin{lemma}
\label{curve gamma}The set $\Gamma$ is formed by $5$ Jordan arcs :

\begin{itemize}
\item two curves $\Gamma_{1}^{\pm}$ emerging from $z=1,$ and diverging
respectively to infinity in $%
%TCIMACRO{\U{2102} }%
%BeginExpansion
\mathbb{C}
%EndExpansion
_{\pm};$

\item two curves $\Gamma_{-1}^{\pm}$ emerging from $z=-1,$ and diverging
respectively to infinity in $%
%TCIMACRO{\U{2102} }%
%BeginExpansion
\mathbb{C}
%EndExpansion
_{\pm}.$
\end{itemize}
\end{lemma}

From Lemma \ref{curve gamma}, $\Gamma$ splits $%
%TCIMACRO{\U{2102} }%
%BeginExpansion
\mathbb{C}
%EndExpansion
$ into 4 connected domains :

\begin{itemize}
\item $\Omega_{1}$ limited by $\Gamma_{1}^{\pm}$ and containing $z=2;$

\item $\Omega_{2}$ limited by $\Gamma_{-1}^{\pm}$ and containing $z=-2;$

\item $\Omega_{+}$ $\subset%
%TCIMACRO{\U{2102} }%
%BeginExpansion
\mathbb{C}
%EndExpansion
_{+},$ limited by $\left[  -1,1\right]  $ and $\Gamma_{\pm1}^{+};$

\item $\Omega_{-}$ $\subset%
%TCIMACRO{\U{2102} }%
%BeginExpansion
\mathbb{C}
%EndExpansion
_{-},$ limited by $\left[  -1,1\right]  $ and $\Gamma_{\pm1}^{-}.$ See Figure
\ref{2}.
\end{itemize}

\begin{figure}[th]
\centering\includegraphics[height=1.8in,width=2.8in]{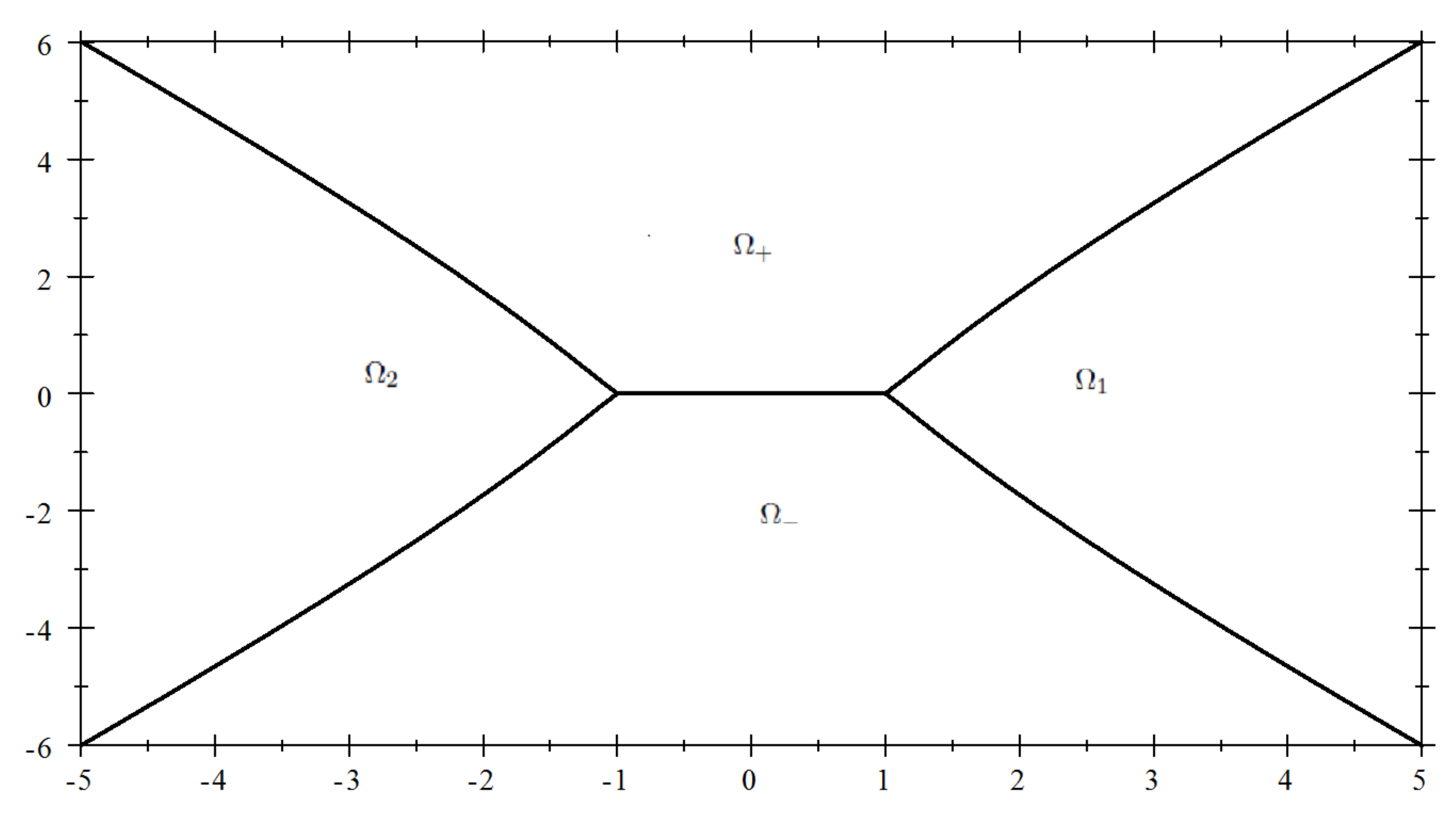}\caption{Approximate
plot of the curve $\Gamma$}%
\label{2}%
\end{figure}

Then main result of this paper is the following :

\begin{proposition}
\label{critical graph} For any complex number $a\in%
%TCIMACRO{\U{2102} }%
%BeginExpansion
\mathbb{C}
%EndExpansion
^{+},$ the segment $\left[  -1,1\right]  $ is a short trajectory of the
quadratic differential $\varpi.$ Moreover,

\begin{itemize}
\item[(i)] For $a\in\Omega_{1}\cup\Omega_{2},$ $\varpi$ has an unbrowken short
trajectory connecting $a$ and $\overline{a}$. $\Gamma_{p}$ splits $\overline{%
%TCIMACRO{\U{2102} }%
%BeginExpansion
\mathbb{C}
%EndExpansion
}$ into six half-plane domains and one strip domain; see Figure
\ref{3}.

\item[(ii)] For $a$ $\in\Gamma_{\pm1},$ $\varpi$ has two short trajectories
connecting $\pm1$ to $a$ and to $\overline{a}.$ $\Gamma_{p}$ splits
$\overline{%
%TCIMACRO{\U{2102} }%
%BeginExpansion
\mathbb{C}
%EndExpansion
}$ into six half-plane domains; see Figure \ref{4}.

\item[(iii)] For $a$ $\in\Omega_{+}\cup\Omega_{-},$ there is no short
trajectory connecting $a$ and $\overline{a}$. $\Gamma_{p}$ splits $\overline{%
%TCIMACRO{\U{2102} }%
%BeginExpansion
\mathbb{C}
%EndExpansion
}$ into six half-plane domains and two strip domains; see Figure
\ref{5}.
\end{itemize}
\end{proposition}
\begin{figure}[th]
\centering\includegraphics[height=1.8in,width=2.8in]{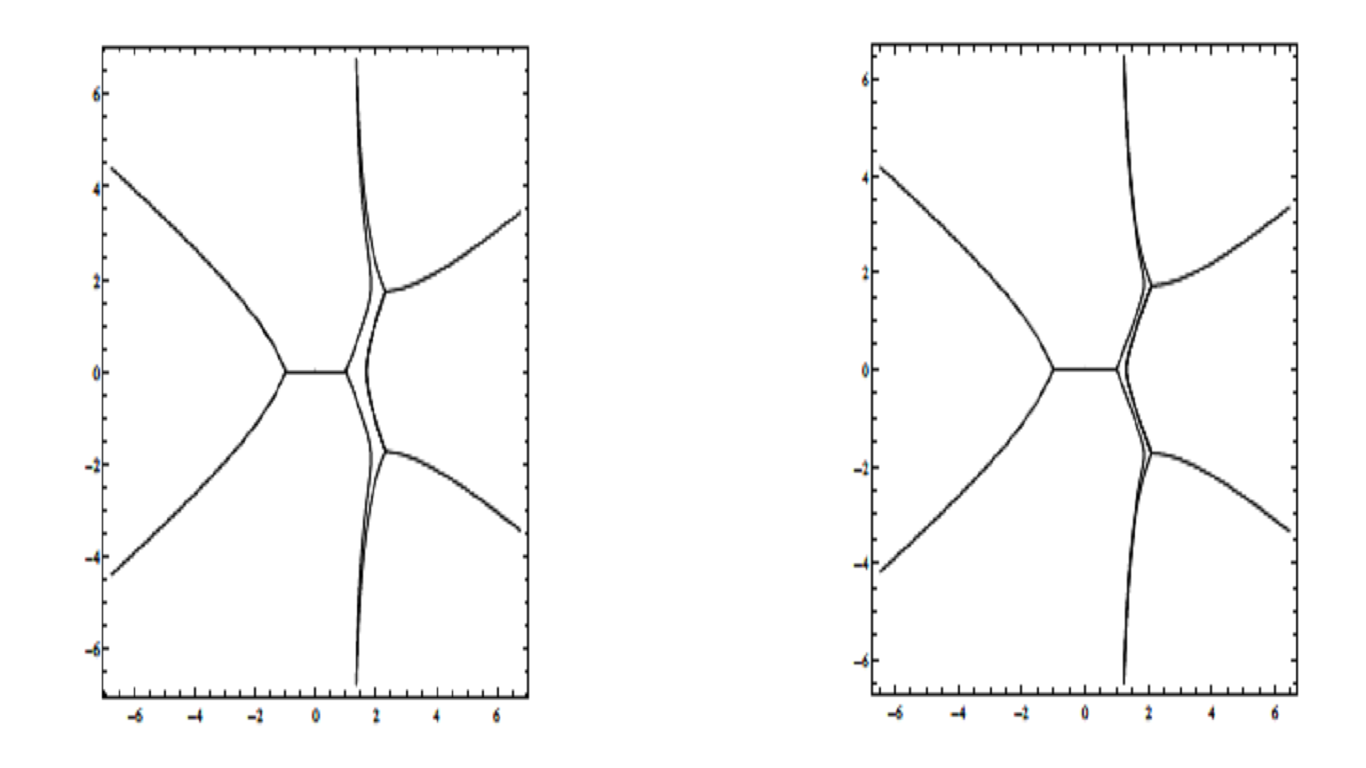}\caption{Structure of the trajectories near a simple zero (left),double zero (center), and pole of order 8 (right)}%
\label{3}%
\end{figure}
\begin{figure}[th]
\centering\includegraphics[height=1.8in,width=2.8in]{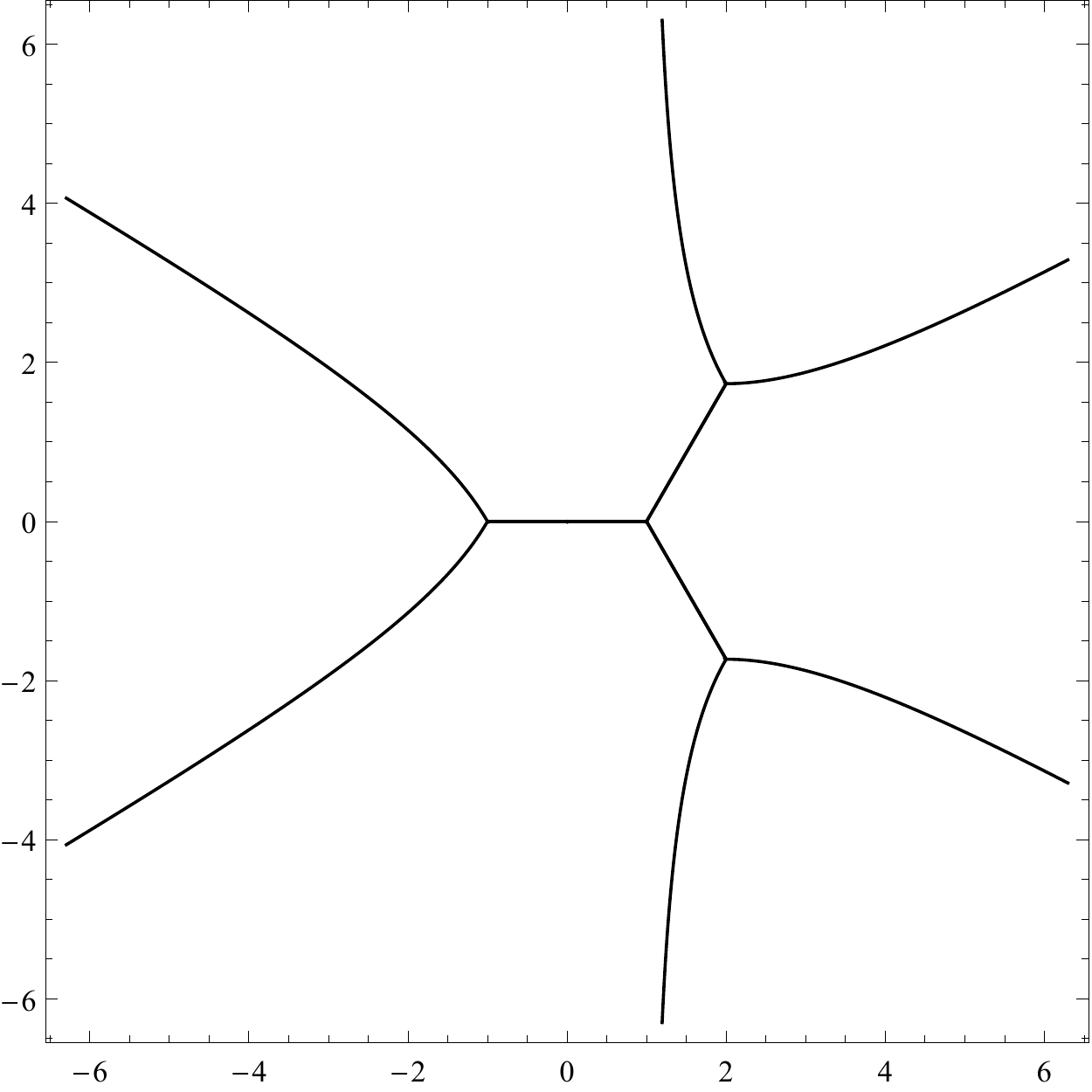}\caption{(center), and pole of order 8 (right)}%
\label{4}%
\end{figure}

\begin{figure}[th]
\centering\includegraphics[height=1.8in,width=2.8in]{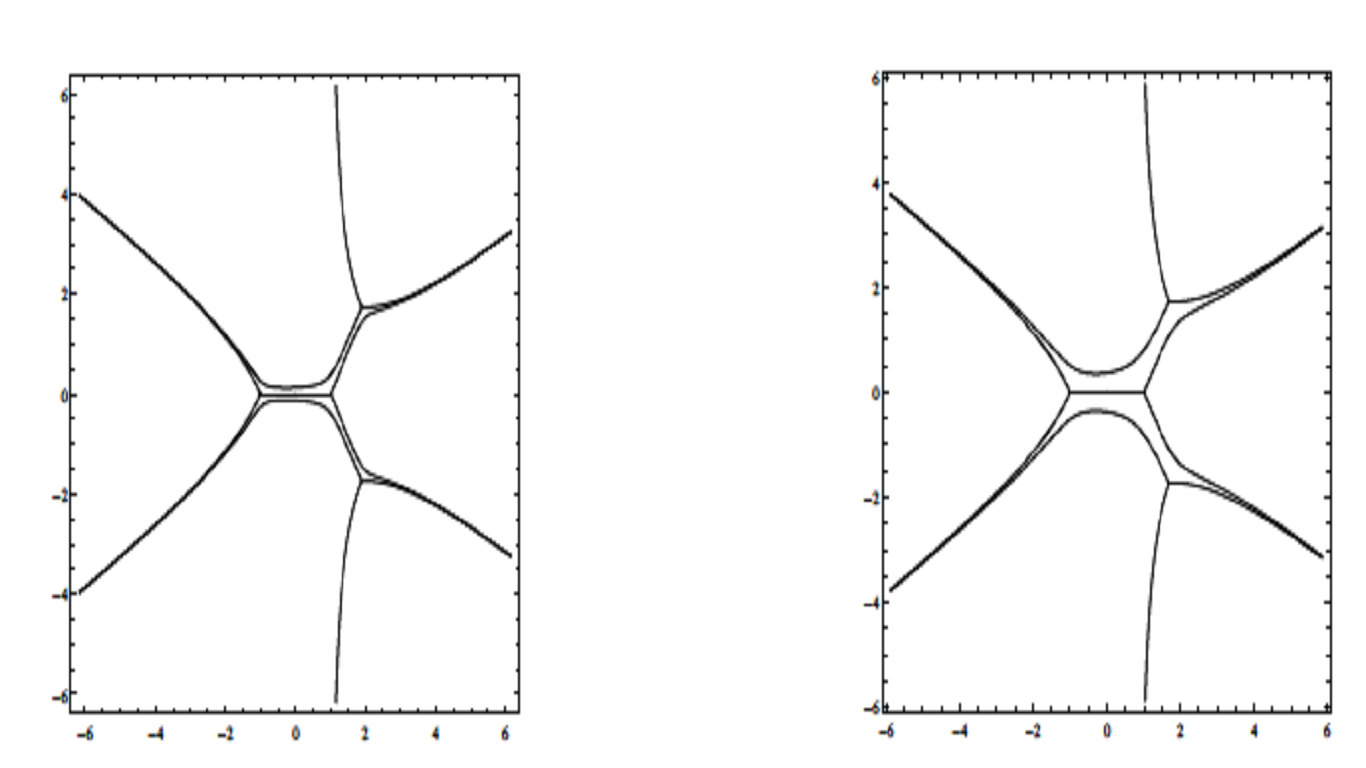}\caption{Structure of the trajectories near a simple zero (left),double zero (center), and pole of order 8 (right)}%
\label{5}%
\end{figure}

\bigskip\bigskip

\begin{remark}

\begin{enumerate}
\item \bigskip If $p\in%
%TCIMACRO{\U{2102} }%
%BeginExpansion
\mathbb{C}
%EndExpansion
\left[  X\right]  $ and $\varpi$ has no short trajectory, then $\Gamma_{p}$
splits $\overline{%
%TCIMACRO{\U{2102} }%
%BeginExpansion
\mathbb{C}
%EndExpansion
}$ into six half-plane domains and three strip domains. See Figure \ref{10}.

\item If $p\in%
%TCIMACRO{\U{211d} }%
%BeginExpansion
\mathbb{R}
%EndExpansion
\left[  X\right]  ,$ then the critical graph $\Gamma_{p}$ is symmetric with
respect to the real axis. It is obvious that when $p=$ $\prod_{i=1}^{4}\left(
z-a_{i}\right)  $ with simple real zeros : $a_{1}<a_{2}<a_{3}<a_{4},$ then the
segments $\left[  a_{1},a_{2}\right]  $ and $\left[  a_{3},a_{4}\right]  $ are
two short trajectories of $\varpi_{p}$. See Figure \ref{6}.
\end{enumerate}
\end{remark}

\begin{figure}[th]
\centering\includegraphics[height=1.8in,width=2.8in]{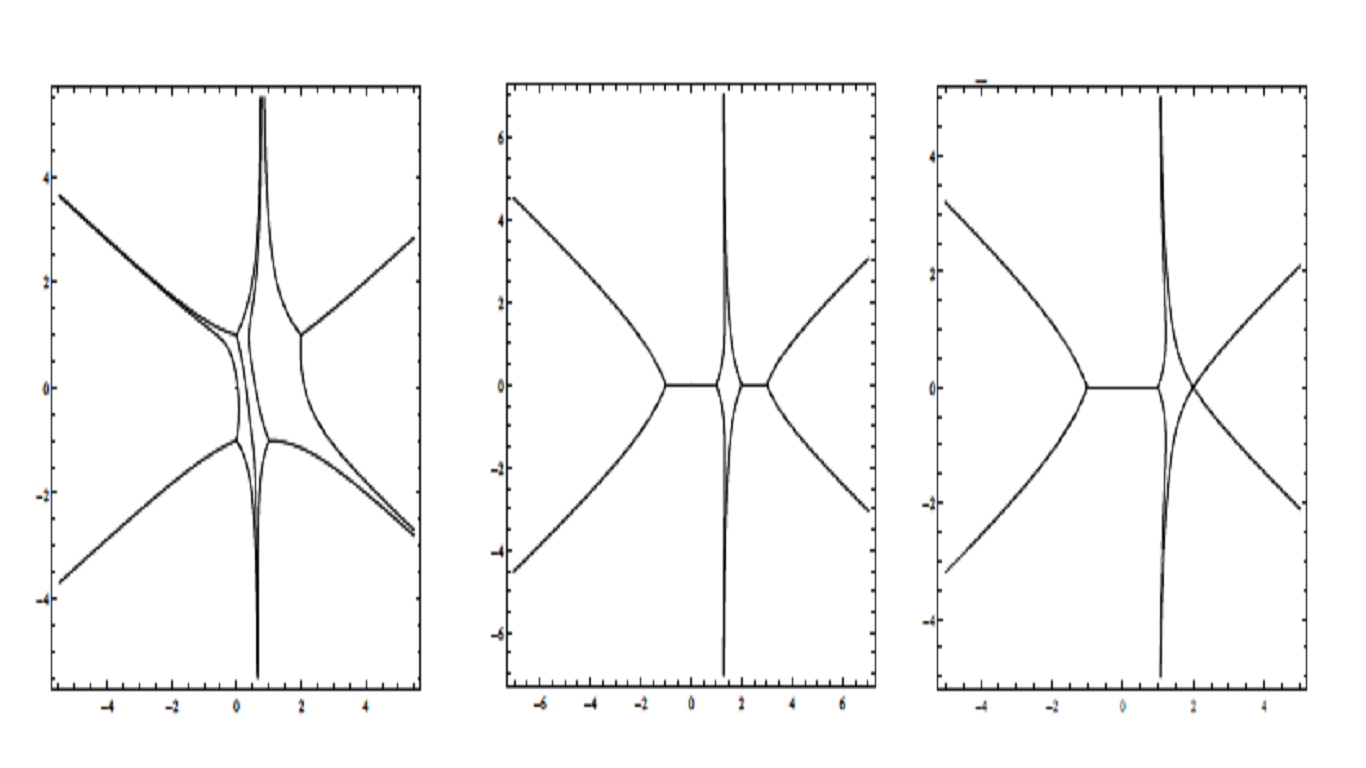}\caption{Structure of the trajectories near a simple zero (left),double zero (center), and pole of order 8 (right)}%
\label{6}%
\end{figure}\bigskip

\subsection{\bigskip The algebraic equation \label{eqal}}

The Cauchy transform $\mathcal{C}_{\nu}$ of a compactly supported Borelian
complex measure $\nu$ is the analytic function defined in $%
%TCIMACRO{\U{2102} }%
%BeginExpansion
\mathbb{C}
%EndExpansion
\setminus$\emph{supp}$\left(  \nu\right)  $ by :
\begin{equation}
\mathcal{C}_{\nu}\left(  z\right)  =\int_{%
%TCIMACRO{\U{2102} }%
%BeginExpansion
\mathbb{C}
%EndExpansion
}\frac{d\nu\left(  t\right)  }{z-t}. \label{cauchy}%
\end{equation}
It satisfies:%
\begin{equation}
\mathcal{C}_{\nu}\left(  z\right)  =\frac{\nu\left(
%TCIMACRO{\U{2102} }%
%BeginExpansion
\mathbb{C}
%EndExpansion
\right)  }{z}+\mathcal{\allowbreak O}\left(  z^{-2}\right)  ,z\rightarrow
\infty, \label{cauchy cond 1}%
\end{equation}
and the inversion formula (which should be understood in the distributions
sense):%
\[
\nu=\frac{1}{\pi}\frac{\partial\mathcal{C}_{\nu}}{\partial\overline{z}}.\text{
}%
\]

In particular, the normalized root-counting measure $\nu_{n}=\nu(P_{n})$ of a
complex polynomial $P_{n}$ of a complex polynomial $P_{n}$ of degree $n$ is
defined for any compact set $K$ in $%
%TCIMACRO{\U{2102} }%
%BeginExpansion
\mathbb{C}
%EndExpansion
$ by :
\[
\int_{K}d\nu_{n}=\frac{\text{\emph{number of zeros} \emph{of} }P_{n}\text{
\emph{in} }K}{n};
\]
the Cauchy transform of $\nu_{n}$ is :
\[
\mathcal{C}_{\nu_{n}}(z)=\int_{%
%TCIMACRO{\U{2102} }%
%BeginExpansion
\mathbb{C}
%EndExpansion
}\frac{d\nu_{n}\left(  t\right)  }{z-t}=\frac{P_{n}^{^{\prime}}(z)}{nP_{n}%
(z)};P_{n}(z)\neq0.
\]

Let us consider the following algebraic equation :%
\begin{equation}
\mathcal{C}^{2}\left(  z\right)  +r\left(  z\right)  \mathcal{C}\left(
z\right)  +s\left(  z\right)  =0. \label{eq alg}%
\end{equation}
In this subsection, we give the connection between investigate the existence
of a compactly supported Borelian signed measure $\nu,$ whose Cauchy transform
$\mathcal{C}_{\nu}$ satisfies almost everywhere in the complex plane $%
%TCIMACRO{\U{2102} }%
%BeginExpansion
\mathbb{C}
%EndExpansion
$ where $r$ and $s$ are two polynomials with degree $2$ and at most $1;$
\begin{align*}
r\left(  z\right)   &  =az^{2}+bz+c;a\in%
%TCIMACRO{\U{2102} }%
%BeginExpansion
\mathbb{C}
%EndExpansion
\setminus\left\{  0\right\}  ,b,c\in%
%TCIMACRO{\U{2102}}%
%BeginExpansion
\mathbb{C}%
%EndExpansion
\\
s\left(  z\right)   &  =ez+f;e,f\in%
%TCIMACRO{\U{2102} }%
%BeginExpansion
\mathbb{C}
%EndExpansion
.
\end{align*}
The measure $\nu$ (if exists) is called \emph{real mother-body measure} of the
equation (\ref{eq alg}).

The following proposition gives the connection between ..

\begin{proposition}
\label{connection}If equation (\ref{eq alg}) admits a real mother-body measure
$\nu$, then:

\begin{itemize}
\item $\frac{e}{a}\in%
%TCIMACRO{\U{211d} }%
%BeginExpansion
\mathbb{R}
%EndExpansion
,$

\item any connected curve in the support of $\nu$ coincides with a short
trajectory of the quadratic differential
\begin{equation}
\varpi=-Q\left(  z\right)  dz^{2}, \label{1qd}%
\end{equation}
where $Q\left(  z\right)  =r^{2}\left(  z\right)  -4s\left(  z\right)  $ is
the discriminant of the quadratic equation (\ref{eq alg}).
\end{itemize}
\end{proposition}

\section{Proofs}

\begin{proof}
[Proof of Lemma \ref{residue}]Since $p$ is a real polynomial, then
$\overline{\sqrt{p\left(  t\right)  }}=\sqrt{p\left(  \overline{t}\right)  },$
and get, after the change of variable $u=\overline{t}$ in second integral :
\begin{align*}
\Re\left(  \int_{\overline{z}}^{z}\sqrt{p\left(  t\right)  }dt\right)   &
=\Re\left(  \int_{1}^{z}\sqrt{p\left(  t\right)  }dt-\int_{1}^{\overline{z}%
}\sqrt{p\left(  t\right)  }dt\right) \\
&  =\Re\left(  \int_{1}^{z}\sqrt{p\left(  t\right)  }dt-\overline{\int_{1}%
^{z}\sqrt{p\left(  t\right)  }dt}\right) \\
&  =\Re\left(  2i\Im\left(  \int_{1}^{z}\sqrt{p\left(  t\right)  }dt\right)
\right) \\
&  =0.
\end{align*}
Let us give a necessary condition to get two short trajectories joining two
different pairs of zeros of $p$ in the general case. We write%
\[
p\left(  z\right)  =z^{4}+\alpha z^{3}+\beta z^{2}+\gamma z+\delta\in%
%TCIMACRO{\U{2102} }%
%BeginExpansion
\mathbb{C}
%EndExpansion
\left[  X\right]  .
\]
and consider two disjoint oriented Jordan arcs $\gamma_{1}$ and $\gamma_{2}$
joining two different pairs zeros. We define the single-valued function
$\sqrt{p\left(  z\right)  }$ in $%
%TCIMACRO{\U{2102} }%
%BeginExpansion
\mathbb{C}
%EndExpansion
\setminus\left(  \gamma_{1}\cup\gamma_{2}\right)  $ with condition
$\sqrt{p\left(  z\right)  }\thicksim z^{2},z\rightarrow\infty.$

From the Laurent expansion at $\infty$ of $\sqrt{p\left(  z\right)  }:$%
\[
\sqrt{p\left(  z\right)  }=z^{2}+\frac{\alpha}{2}z-\frac{\alpha^{2}-4\beta}%
{8}+\frac{\alpha^{3}-4\alpha\beta+8\gamma}{16}z^{-1}+\mathcal{\allowbreak
O}\left(  z^{-2}\right)  ,z\rightarrow\infty,
\]
we deduce the residue%
\[
-res_{\infty}\left(  \sqrt{p\left(  z\right)  }\right)  =\allowbreak\frac
{1}{16}\left(  \alpha^{3}-4\alpha\beta+8\gamma\right)  .
\]
For $s\in\gamma_{1}\cup\gamma_{2},$ we denote by $\left(  \sqrt{p\left(
s\right)  }\right)  _{+}$ and $\left(  \sqrt{p\left(  s\right)  }\right)
_{-}$the limits from the $+$ and $-$ sides, respectively. (As usual, the $+$
side of an oriented curve lies to the left and the $-$ side lies to the right,
if one traverses the curve according to its orientation.) Let%
\[
I=\int_{\gamma_{1}}\left(  \sqrt{p\left(  s\right)  }\right)  _{+}%
ds+\int_{\gamma_{2}}\left(  \sqrt{p\left(  s\right)  }\right)  _{+}ds.
\]
Since
\[
\left(  \sqrt{p\left(  s\right)  }\right)  _{+}=-\left(  \sqrt{p\left(
s\right)  }\right)  _{-},s\in\gamma_{1}\cup\gamma_{2},
\]
we have
\[
2I=\int_{\gamma_{1}\cup\gamma_{2}}\left[  \left(  \sqrt{p\left(  s\right)
}\right)  _{+}-\left(  \sqrt{p\left(  s\right)  }\right)  _{-}\right]
ds=\oint_{\Gamma}\sqrt{p\left(  z\right)  }dz,
\]
where $\Gamma$ is a closed contours encircling the $\gamma_{1}$ and
$\gamma_{2}$. After the contour deformation we pick up residue at $z=\infty.$
We get, for any choice of the square roots
\begin{align*}
I  &  =\frac{1}{2}\oint_{\Gamma}\sqrt{p\left(  z\right)  }dz=\pm i\pi
res_{\infty}\left(  \sqrt{p\left(  z\right)  }\right) \\
&  =\pm\frac{\pi i}{16}\left(  \alpha^{3}-4\alpha\beta+8\gamma\right)  .
\end{align*}
We conclude the necessary condition :
\[
\Im\left(  \alpha^{3}-4\alpha\beta+8\gamma\right)  =0.
\]

\end{proof}

\begin{proof}
[Proof of Lemma \ref{conse teich1}]

\begin{itemize}
\item[(i)] It is straightforward that for $s\in\left[  -1,1\right]  ,$ we have
$\Re\int_{-1}^{s}\sqrt{p\left(  t\right)  }dt=0.$

\item[(ii)] Suppose that $\gamma_{1}$ and $\gamma_{2}$ are two such
trajectories emanating from the same zero $z_{j}$, spacing with angle
$t\in\left\{  2\pi/3,4\pi/3\right\}  .$ Consider the $\varpi_{p}$-polygon with
edges $\gamma_{1}$ and $\gamma_{2},$ and vertices $z_{j},$ and infinity. The
right side of (\ref{Teich equality}) can take only the values $0,$ or $-1$
while the left side is at list $2$; a contradiction.

\item[(iii)] If a critical trajectory emanating from $z=1$ diverges to
$\infty$ in the direction $D_{2},$ then, also, a critical trajectory emanating
from $z=-1$ must diverge to $\infty$ in the direction $D_{2};$ these two
critical trajectories will form with the segment $\left[  -1,1\right]  $ and
$\infty$ an $\varpi_{p}$-polygon and, again, this violates Lemma
\ref{teich lemma}.

\item[(iv)] The result follows by combining (ii) and (iii), and using Lemma
\ref{teich lemma}.

\item[(v)] The case of four short trajectories can be discarded immediately by
Lemma \ref{teich lemma}.
\end{itemize}

It is an immediate consequence from \ref{teich lemma} that it cannot happen
that two short trajectories connect two zeros of a holomorphic quadratic
differential on the Riemann sphere.
\end{proof}

\begin{proof}
[Proof of Lemma \ref{curve gamma}]The first observation is that :%
\begin{equation}
\Gamma\cap\left\{  z\in%
%TCIMACRO{\U{2102} }%
%BeginExpansion
\mathbb{C}
%EndExpansion
\mid\Re\left(  z\right)  =1,\Im\left(  z\right)  \geq0\right\}  =\left\{
1\right\}  ; \label{first remark}%
\end{equation}
indeed, if for some $y>0,z=1+iy\in\Gamma,$ then
\begin{align*}
0  &  =\Re\left(  \int_{1}^{1+iy}\sqrt{\left(  u-\left(  1+iy\right)  \right)
\left(  u-\left(  1-iy\right)  \right)  \left(  u^{2}-1\right)  }du\right) \\
&  =y^{\frac{3}{2}}\int_{0}^{1}\sqrt{ty\left(  1-t^{2}\right)  }\Re\left(
\sqrt{ty-2i}\right)  dt,u=1+ity,0\leq t\leq1.
\end{align*}
However, if we choose the argument in $\left[  0,2\pi\right[  ,$ then for any
$t\in\left]  0,1\right[  ,$ we have
\begin{align*}
\frac{3\pi}{2}  &  <\arg\left(  ty-2i\right)  <2\pi\Longrightarrow\frac{3\pi
}{4}<\arg\sqrt{ty-2i}<\pi\\
&  \Longrightarrow\sqrt{ty\left(  1-t^{2}\right)  }\Re\left(  \sqrt
{ty-2i}\right)  <0\\
&  \Longrightarrow\int_{0}^{1}\sqrt{ty\left(  1-t^{2}\right)  }\Re\left(
\sqrt{ty-2i}\right)  dt<0,
\end{align*}
which gives a contradiction.

In order to prove that $\Gamma$ is a curve, and since we know that $\left[
-1,1\right]  \subset\Gamma,$ with the observation (\ref{first remark}), it is
sufficient to prove the result in the set%
\[
\Pi=\left\{  \left(  x,y\right)  ;x>1,y>0\right\}  ,
\]
and the real functions $F$ and $G$ defined for $\left(  x,y\right)  $ in $\Pi$
by:
\begin{align*}
F\left(  x,y\right)   &  =\Re\left(  \int_{0}^{x}\sqrt{\left(  u-\left(
x+iy\right)  \right)  \left(  u-\left(  x-iy\right)  \right)  \left(
u^{2}-1\right)  }du\right) \\
&  =\int_{1}^{x}\sqrt{\left(  \left(  u-x\right)  ^{2}+y^{2}\right)  \left(
u^{2}-1\right)  }du,\\
G\left(  x,y\right)   &  =\Re\int_{x}^{x+iy}\sqrt{\left(  u-\left(
x+iy\right)  \right)  \left(  u-\left(  x-iy\right)  \right)  \left(
u^{2}-1\right)  }du.
\end{align*}
In other, the set $\Gamma$ is defined by :%
\[
\left\{  \left(  x,y\in\Pi\mid\left(  F+G\right)  \left(  x,y\right)  \right)
=0\right\}  .
\]
Observe that $\left(  F+G\right)  $ is differentiable in $\Pi,$ and $z=1$ is
the only singular point of $F+G.$

We have
\[
\frac{\partial F}{\partial x}\left(  x,y\right)  =\sqrt{y^{2}\left(
x^{2}-1\right)  }+\Re\left(  \int_{1}^{x}\frac{\left(  x-t\right)  \left(
t^{2}-1\right)  }{\sqrt{\left(  \left(  t-x\right)  ^{2}+y^{2}\right)  \left(
t^{2}-1\right)  }}dt\right)  >0.
\]
By the other hand, with the change of variable $u=x+ity,0\leq t\leq1,$ we get
:
\begin{align*}
\frac{\partial G}{\partial x}\left(  x,y\right)   &  =\frac{\partial}{\partial
x}\left[  \Re\int_{0}^{1}iy^{2}\sqrt{\left(  1-t^{2}\right)  \left(  \left(
x+ity\right)  ^{2}-1\right)  }dt\right] \\
&  =\Re\int_{0}^{1}iy^{2}\left(  x+ity\right)  \frac{1-t^{2}}{\sqrt{\left(
\left(  x+ity\right)  ^{2}-1\right)  \left(  1-t^{2}\right)  }}dt\\
&  =-y^{2}\int_{0}^{1}\sqrt{1-t^{2}}\Im\left(  \frac{z_{t}}{\sqrt{z_{t}^{2}%
-1}}\right)  dt,\text{ where }z_{t}=x+ity.
\end{align*}
We have for any $t\in\left[  0,1\right]  $ :
\begin{align*}
0  &  \leq\arg\left(  z_{t}\right)  <\frac{\pi}{2}\Longrightarrow0\leq
\arg\left(  z_{t}\right)  \leq\arg\left(  \sqrt{z_{t}^{2}-1}\right)  \leq
\frac{\pi}{2}\\
&  \Longrightarrow-\frac{\pi}{2}\leq\arg\left(  \frac{z_{t}}{\sqrt{z_{t}%
^{2}-1}}\right)  =\arg\left(  z_{t}\right)  -\arg\left(  \sqrt{z_{t}^{2}%
-1}\right)  \leq0\\
&  \Longrightarrow\sqrt{1-t^{2}}\Im\left(  \frac{z_{t}}{\sqrt{z_{t}^{2}-1}%
}\right)  \leq0\text{ (non identically vanishing).}%
\end{align*}
It follows that
\[
\frac{\partial G}{\partial x}\left(  x,y\right)  >0.
\]
We conclude that the set $\Gamma$ is a regular curve in $%
%TCIMACRO{\U{2102} }%
%BeginExpansion
\mathbb{C}
%EndExpansion
$ by $\allowbreak$applying the Implicit Function Theorem to the function $F+G$.
\end{proof}

\begin{proof}
[Proof of Lemma \ref{asymptotic gamma}]

From the power series expansion
\[
\sqrt{p\left(  t\right)  }=\allowbreak\sqrt{\left(  \overline{z}-1\right)
\left(  2z-2\right)  }\sqrt{t-1}+\mathcal{O}\left(  t-1\right)  ^{\frac{3}{2}%
},t\rightarrow1,
\]
we get for $z\rightarrow1,z\in\Gamma\cap%
%TCIMACRO{\U{2102} }%
%BeginExpansion
\mathbb{C}
%EndExpansion
_{+}^{+}$
\begin{align*}
0  &  =\Re\int_{1}^{z}\sqrt{\left(  t-z\right)  \left(  t-\overline{z}\right)
\left(  t^{2}-1\right)  }dt\\
&  =\Re\left(  \frac{2\sqrt{2}}{3}\left\vert z-1\right\vert \left(
z-1\right)  ^{\frac{3}{2}}\right) \\
&  =\frac{2\sqrt{2}}{3}\left\vert z-1\right\vert \Re\left(  \left(
z-1\right)  ^{\frac{3}{2}}\right)  .
\end{align*}
It follows that
\[
\arg\left(  z-1\right)  ^{3}\equiv\pi\operatorname{mod}\left(  2\pi\right)
,z\rightarrow1,z\in\Gamma\cap%
%TCIMACRO{\U{2102} }%
%BeginExpansion
\mathbb{C}
%EndExpansion
_{+}^{+},
\]
which gives the local structure of $\Gamma$ near $z=1.$

Writing $z=re^{ix}\in\Gamma\cap%
%TCIMACRO{\U{2102} }%
%BeginExpansion
\mathbb{C}
%EndExpansion
_{+}^{+}$ with $r>0,x\in\left(  2,\pi\right)  ,$ in the integral defining
$\Gamma,$ we get with the change of variable $s=tre^{ix},t\in\left[
0,1\right]  $ :
\begin{equation}
\Re\int_{0}^{1}e^{ix}\sqrt{\left(  1-t\right)  \left(  1-te^{2ix}\right)
\left(  e^{2ix}t^{2}-\frac{1}{r^{2}}\right)  }=0. \label{r}%
\end{equation}
Taking the limits when $r\rightarrow\infty,$ equality (\ref{r}) becomes%
\[
\Re\left(  e^{3ix}\int_{0}^{1}t\sqrt{\left(  t-1\right)  \left(
t-e^{-2ix}\right)  }dt\right)  =0.
\]
With the change of variable $t=\alpha u+\beta,$ where%
\[
\alpha=\frac{1-e^{-2ix}}{2},\beta=\frac{1+e^{-2ix}}{2},
\]
we get%
\begin{align*}
0  &  =\Re\int_{i\cot x}^{1}\alpha^{2}e^{3ix}\left(  \alpha u+\beta\right)
\sqrt{\left(  u-\frac{1-\beta}{\alpha}\right)  \left(  u-\frac{e^{-2ix}-\beta
}{\alpha}\right)  }du\\
&  =\Re\left(  \alpha^{3}e^{3ix}\int_{i\cot x}^{1}\left(  u-i\cot x\right)
\sqrt{u^{2}-1}du\right) \\
&  =\Im\left(  \int_{i\cot x}^{1}\left(  t-i\cot x\right)  \sqrt{t^{2}%
-1}\right)  \sin^{3}x=0.
\end{align*}
It follows that%
\[
\Im\left(  \int_{i\cot x}^{0}\left(  t-i\cot x\right)  \sqrt{t^{2}-1}\right)
+\Im\left(  \int_{0}^{1}\left(  t-i\cot x\right)  \sqrt{t^{2}-1}\right)  =0,
\]
and then%
\begin{align*}
0  &  =\int_{0}^{\cot x}t\sqrt{t^{2}+1}-\cot x\int_{0}^{\cot x}\sqrt{t^{2}%
+1}+\int_{0}^{1}t\sqrt{1-t^{2}}\\
&  =\frac{1}{3\sin^{3}x}-\frac{1}{2}\cot x\left(  \frac{\cot x}{\sin x}%
+\arg\sinh\left(  \cot x\right)  \right)  ,
\end{align*}
which gives equation (\ref{asympt en l'inf}).

To prove the existence and uniqueness of solution of equation
(\ref{asympt en l'inf}) in $\left(  0,\frac{\pi}{2}\right)  ,$ we need to
study the function%
\[
x\overset{f}{\longmapsto}\sinh\left(  \frac{2}{3\cot x\sin^{3}x}-\frac{\cot
x}{\sin x}\right)  -\cot x;x\in\left(  0,\frac{\pi}{2}\right)  .
\]
It order to show that $f^{\prime}>0,$ it is sufficient to show that%
\[
\frac{2}{3\cos^{2}x}-\sin^{5}x>0;x\in\left(  0,\frac{\pi}{2}\right)
\allowbreak,
\]
which is straightforward by studying the rational function
\[
x\longmapsto\frac{2}{3\left(  1-t^{2}\right)  }-t^{5}=\frac{1}{3\left(
1-t^{2}\right)  }\left[  2-3\left(  1-t^{2}\right)  t^{5}\right]  ;t\in\left(
0,1\right)  .
\]

\end{proof}

\begin{proof}
[Proof of Proposition \ref{critical graph}]

\begin{enumerate}
\item[i)] Let us begin our proof by the first observation, that, if for some
$a,\gamma_{1}$ diverges to $\infty$ in the direction $D_{0},$ then there is no
short trajectory connecting $a$ and $\overline{a}$.????

Suppose that for some $a=x+iy\in\Omega_{1},$with $x>1,y>0,$ there is no short
trajectory connecting $a$ and $\overline{a}.$ Obviously, all critical
trajectories emanating from $a$ stay in the upper half plane. Let us denote by
$\gamma_{1}$ the only critical trajectory emanating from $z=1$ and diverging
to $\infty$ in the upper half plane. From Lemma \ref{conse teich1},
$\gamma_{1}$ follows a certain direction $D_{i},$ $i\in\left\{  0,1\right\}
$. Then, we claim that $i=0;$ otherwise, if $i=1,$ then, at most two critical
trajectories emanating from $a$ must diverge to $\infty$ in the same
direction, which contradicts Lemma \ref{conse teich1} again.

We denote by $a_{t}=x+ity,$ $t\geq0,$ and consider the set%
\[
N_{x}=\left\{  t>0\mid\text{there is no short trajectory connecting }%
a_{t}\text{ and }\overline{a_{t}}\right\}  .
\]

Since for any $t\in N_{x},$ $\gamma_{1}$ diverges to $\infty$ in the direction
$D_{0},$ by continuity of the trajectories (in the Hausdorff metric), we
conclude that $\inf\left(  N_{x}\right)  =t_{0}>0,$ (otherwise, the segment
$\left[  1,x\right]  $ will be a short trajectory, which contradicts
(\ref{cond necess}).then, for some $0<t<<1$ for s set $a.$

\item[ii)] If there is no short trajectory connecting $1$ and $a$ for some
$a\in\Gamma,$ then there is a critical trajectory $\gamma_{a}$ that emanates
from $a$ and diverges to $\infty$ in the same direction of $\gamma_{1}.$ From
the behaviour of orthogonal trajectories at $\infty,$ we take an orthogonal
trajectory $\sigma$ that hits $\gamma_{a}$ and $\gamma_{1}$ respectively in
two points $b$ and $c.$ (there are infinitely many such orthogonal
trajectories $\sigma$ ) We consider a path $\gamma$ connecting $z=1$ and $a$,
formed by the part of $\gamma_{1}$ from $z=1$ to $b$, the part of $\sigma$
from $b$ to $c;$ and the part of $\gamma_{a}$ from $c$ to $a$the critical
trajectory intersecting $\Gamma$ from $c$ to $a.$ Then
\begin{align*}
\Re\int_{\gamma}\sqrt{p\left(  t\right)  }dt  &  =\Re\int_{1}^{b}%
\sqrt{p\left(  t\right)  }dt+\Re\int_{b}^{c}\sqrt{p\left(  t\right)  }%
dt+\Re\int_{c}^{a}\sqrt{p\left(  t\right)  }dt\\
&  =\Re\int_{b}^{c}\sqrt{p\left(  t\right)  }dt\neq0,
\end{align*}
which contradicts the fact that $a\in\Gamma.$

\item[iii)] vv
\end{enumerate}
\end{proof}

\bigskip

\begin{proof}
[Proof of Proposition \ref{connection}]Solutions of the quadratic equation
(\ref{eq alg}) are%
\[
\mathcal{C}^{\pm}\left(  z\right)  =-r\left(  z\right)  \pm\sqrt{r^{2}\left(
z\right)  -4s\left(  z\right)  },
\]
with some choice of the square root. If $\nu$ is a real mother-body measure of
(\ref{eq alg}), then, condition (\ref{cauchy cond 1}) implies that
\[
\mathcal{C}_{\nu}\left(  z\right)  =-r\left(  z\right)  +\sqrt{r^{2}\left(
z\right)  -4s\left(  z\right)  }=-\frac{2e}{a}\frac{1}{z}+\mathcal{O}\left(
z^{-2}\right)  ,z\rightarrow\infty.
\]
we conclude that
\[
\nu\left(
%TCIMACRO{\U{2102} }%
%BeginExpansion
\mathbb{C}
%EndExpansion
\right)  =-\frac{2e}{a}.
\]
From \cite[Theorem 1]{positive}, and since the Cauchy transform $\mathcal{C}%
_{\nu}\left(  z\right)  $ of the positive measure $\nu$ coincides a.e. in $%
%TCIMACRO{\U{2102} }%
%BeginExpansion
\mathbb{C}
%EndExpansion
$ with an algebraic solution of the quadratic equation (\ref{eq alg}), it
follows that the support of this measure is a finite union of semi-analytic
curves and isolated points (see \cite{garnet}) . Let $\gamma$ be a connected
curve in the support of $\nu.$ For $t\in\gamma,$ we have
\[
\mathcal{C}_{\nu}^{+}\left(  t\right)  -\mathcal{C}_{\nu}^{-}\left(  t\right)
=\sqrt{Q\left(  t\right)  }.
\]
From Plemelj-Sokhotsky's formula, we have
\[
\frac{1}{2\pi i}\left(  \mathcal{C}_{\nu}^{+}\left(  t\right)  -\mathcal{C}%
_{\nu}^{-}\left(  t\right)  \right)  dt\in%
%TCIMACRO{\U{211d} }%
%BeginExpansion
\mathbb{R}
%EndExpansion
^{\ast},t\in\gamma,
\]
and then we get%
\[
-Q\left(  t\right)  dt^{2}>0,t\in\gamma,
\]
which shows that $\gamma$ is a horizontal trajectory of the quadratic
differential $-Q\left(  z\right)  dz^{2}$ on the Riemann sphere. For more
details, we refer the reader to \cite{amf rakh1},\cite{pritsker}%
,\cite{Shapiro},\cite{bullgard}.
\end{proof}

In the end, this analysis can be done in the case where $p$ is a real
polynomial without real zeros: Suppose that $p\left(  z\right)  =\left(
z-a\right)  \left(  z-b\right)  \left(  z-\overline{a}\right)  \left(
z-\overline{b}\right)  $ with $a\neq b\in%
%TCIMACRO{\U{2102} }%
%BeginExpansion
\mathbb{C}
%EndExpansion
,\Im\left(  a\right)  ,\Im\left(  a\right)  >0.$ Then the quadratic
differential has at list one short trajectory. Indeed, if no one of the
trajectories emanating from $a$ is short, then, they must diverge all to
$\infty$ in the $3$ different asymptotic directions in $%
%TCIMACRO{\U{2102} }%
%BeginExpansion
\mathbb{C}
%EndExpansion
_{+}$. It follows from Lemma \ref{teich lemma} that at most two trajectories
emanating from $b$ can diverge to $\infty$ in $%
%TCIMACRO{\U{2102} }%
%BeginExpansion
\mathbb{C}
%EndExpansion
_{+},$ and then, with consideration of symmetry, the remaining one is a short
trajectory that connects $b$ and $\overline{b}.$ See Figure \ref{7}.

\begin{figure}[th]
\centering\includegraphics[height=1.8in,width=2.8in]{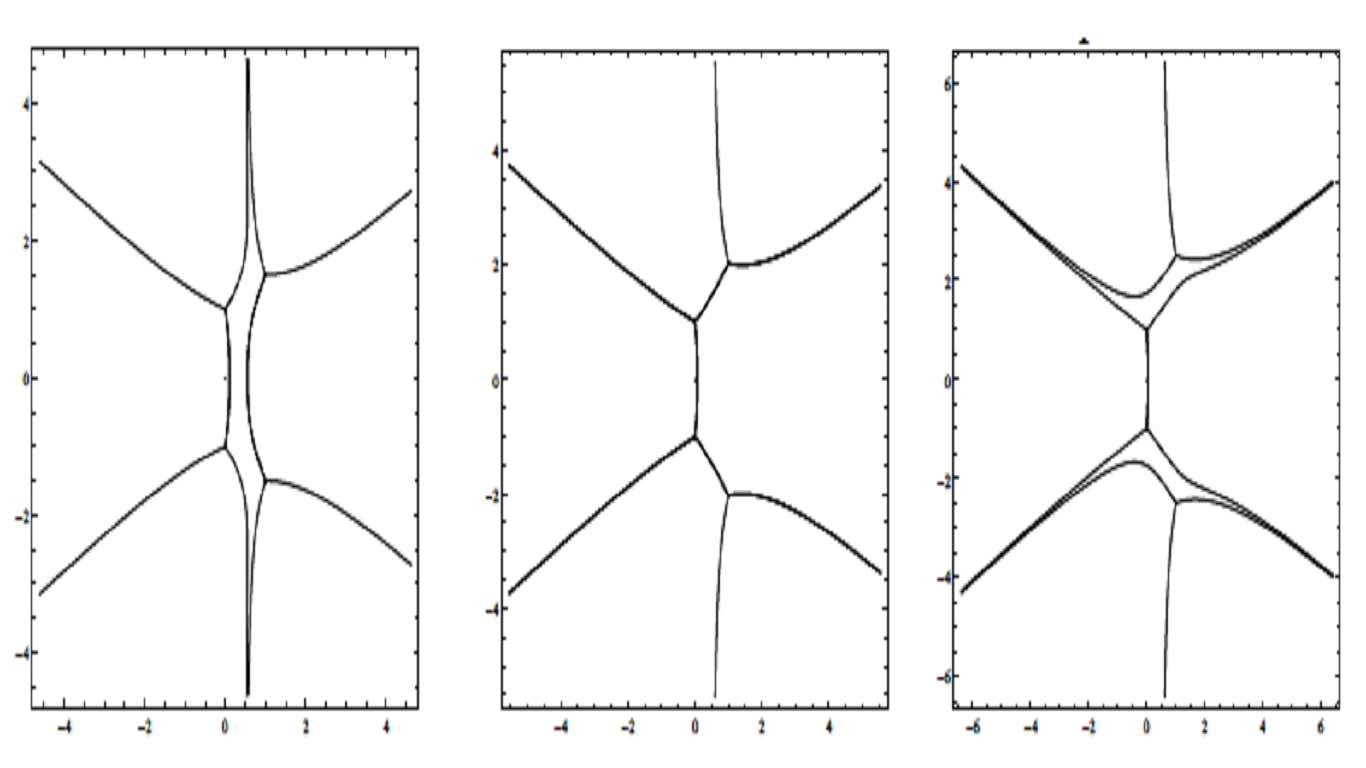}\caption{Structure of the trajectories near a simple zero (left),double zero (center), and pole of order 8 (right)}%
\label{7}%
\end{figure}

\begin{acknowledgement}
\bigskip\ Part of this work was carried out during a visit of F.T to the
university of Stockholm, Sweden. We would to thank Professor Boris Shapiro for
many helpful discussions.
\end{acknowledgement}

%%%%%%%%%%%%%%%%%%%%%%%%%%%%%%%%%%%%%%%%%%%%%%%%%%%%%%%%%%
\bigskip

\texttt{Mondher Chouikhi(chouikhi.mondher@gmail.com)}

\texttt{Faouzi Thabet(faouzithabet@yahoo.fr)}

\texttt{Institut sup\'{e}rieur des sciences appliqu\'{e}es et de technologie }

\texttt{de} \texttt{Gab\`{e}s. Rue Omar Ibn Al khattab 6072. Gab\`{e}s.
Tunisia.}


\begin{thebibliography}{99}                                                                                               %


\bibitem {ref1}Kohei Iwaki,Tomoki Nakanishi; Exact WKB analysis and cluster
algebras. Journal of Physics A Mathematical and Theoretical 47(47)
\textperiodcentered\ January 2014.

\bibitem {ref2}A.Varos.The return of quartic oscillator. The complex WKB
method." Annales de l'I.H.P. Physique th\'{e}orique 39.3 (1983): 211-338.
%TCIMACRO{\TEXTsymbol{<}}%
%BeginExpansion
$<$%
%EndExpansion
http://eudml.org/doc/76217%
%TCIMACRO{\TEXTsymbol{>}}%
%BeginExpansion
$>$%
%EndExpansion
.

\bibitem {ref3}Carl M. Bender and Stefan Boettcher. Quasi-exactly solvable
quartic potential. J. Phys. A: Math. Gen. 31 (1998) L273--L277

\bibitem {ref4}Carl M. Bender. Introduction to $\mathcal{PT}$ -Symmetric
Quantum Theory. Contemporary Physics Vol 46, Iss. 4,2005\textperiodcentered

\bibitem {ref7}C.M.Bender, S. Boettcher, Real spectra in non-hermitian
hamiltonians having PT-symmetry. Phys Rev Lett 1998; 80:5243-5246.

\bibitem {ref8}C.M.Bender,E.J.Weniger, Numerical evideence that the
perturbation expansion for a non-hermitian PT-symmetric hamitolian is
Stieltjes; J.Math.Phys,42:2167-2183. 2001

\bibitem {ref9}P.Dorey, R.Tateo, On the relation between Stokes multipliers
and T-Q systems of conformal field theory. Nucl.Phys.B, 563:573-602. 1999.

\bibitem {ref10}Patrick Dorey, Clare Dunning, Roberto Tateo, Spectral
equivalences, Bethe Ansatz equations, and reality properties in PT-symmetric
quantum mechanics. J.Phys.A34:5679-5704,2001.

\bibitem {ref11}E.Delabaere, F.Pham, Eigenvalues of complex hamiltionians with
PT-symmetry I,II. Phys.Lett.A,250:25-32.1998.

\bibitem {ref5}Boris Shapiro and Milos Tater. On Spectral Asymptotics Of
Quasi-Exactly Solvable Quartic And Yablonskii-Vorob'ev Polynomials,https://arxiv.org/abs/1412.3026.

\bibitem {ref6}A. Eremenko and A. Gabrielov. Quasi-exactly solvable quartic:
elementary integrals and asymptotics. 2011 J. Phys. Math. Theor. 44 312001

\bibitem {BENDER1}C. M. Bender and A. Turbiner. Analytic continuation of
eigenvalue problems. Phys. Lett. A, 173:442--446, 1993.

\bibitem {BENDER4}C. M. Bender and E. J. Weniger, Numerical evidence that the
perturbation expansion for a non-Hermitian PT -symmetric Hamiltonian is
Stieltjes. J. Math. Phys., 42:2167--2183, 2001.

\bibitem {BENDER5}E. Caliceti, S. Graffi and M. Maioli, Perturbation theory of
odd anharmonic oscillators. Comm. Math.Phys., 75:51--66, 1980.

\bibitem {BENDER6}E. Delabaere and F. Pham. Eigenvalues of complex
Hamiltonians with PT-symmetry I, II. Phys. Lett.A, 250:25--32, 1998.

\bibitem {BENDER7}E. Delabaere and D. T. Trinh. Spectral analysis of the
complex cubic oscillator. J. Phys. A: Math. Gen.,33:8771--8796, 2000.

\bibitem {BENDER??}P. Dorey, C. Dunning and R. Tateo. Spectral equivalences,
Bethe ansatz equations, and reality properties in PT -symmetric quantum
mechanics. J. Phys. A: Math. Gen,34:5679--5704, 2001.

\bibitem {shapiro tater 2}Boris Shapiro and Milo\u{s} Tater, Asymptotics and
Monodromy of the Algebraic Spectrum of Quasi-Exactly Solvable Sextic Oscillator....

\bibitem {Atia}M. J. Atia, Andrei Mart\'{\i}nez-Finkelshtein, Pedro
Mart\'{\i}nez-Gonzalez, and F. Thabet, Quadratic differentials and asymptotics
of Laguerre polynomials with varying complex parameters, J. Math. Anal. Appl.
416 (2014).

\bibitem {Atia2}M. J. Atia and F.Thabet, Quadratic differentials A(z-a)(z-b)dz%
%TCIMACRO{\U{b2}}%
%BeginExpansion
${{}^2}$%
%EndExpansion
/(z-c)%
%TCIMACRO{\U{b2} }%
%BeginExpansion
${{}^2}$
%EndExpansion
and algebraic Cauchy transform, Czech.Math.J., 66 (141) (2016), 351--363.

\bibitem {AMF FT}Mart\'{\i}nez-Finkelshtein, A., Mart\'{\i}nez-Gonz\'{a}lez,
P. \& Thabet, F. Comput. Methods Funct. Theory (2016) 16: 347. https://doi.org/10.1007/s40315-015-0146-7

\bibitem {F.Thabet}THABET, F. (2016). ON THE EXISTENCE OF FINITE CRITICAL
TRAJECTORIES IN A FAMILY OF QUADRATIC DIFFERENTIALS. Bulletin of the
Australian Mathematical Society, 94(1), 80-91. doi:10.1017/S000497271600006X.

\bibitem {Shapiro tater}Boris Shapiro and Milo\u{s} Tater, On special
asymptotics of quasi-exactly solvable quartic and Yablonskii-Vorob'ev polynomials.............

\bibitem {Martinez}A. Martinez-Finkelshtein, E. A. Rakhmanov, Critical
measures, quadratic differentials, and weak limits of zeros of Stieltjes
polynomials, Commun. Math. Phys. vol. 302 (2011) 53-111.

\bibitem {Strebel}K.~Strebel, Quadratic differentials, Vol.~5 of Ergebnisse
der Mathematik und ihrer Grenzgebiete (3) [Results in Mathematics and Related
Areas (3)], Springer-Verlag, Berlin, 1984.

\bibitem {jenkins}J.~A. Jenkins, Univalent functions and conformal mapping,
Ergebnisse der Mathematik und ihrer Grenzgebiete. Neue Folge, Heft 18. Reihe:
Moderne Funktionentheorie, Springer-Verlag, Berlin, 1958.

\bibitem {bullgard}T. Bergkvist and H. Rullg\aa rd, On polynomial
eigenfunctions for a class of differential operators. Math. Res. Lett. 9
(2002), 153-171.

\bibitem {amf rakh1}A. Mart\'{\i}nez-Finkelshtein, E. A. Rakhmanov, Critical
measures, quadratic differentials, and weak limits of zeros of Stieltjes
polynomials, Commun. Math. Phys. vol. 302 (2011) 53-111.

\bibitem {sh takemura}Boris Shapiro,Kouichi Takemura, and Milo\u{s} Tater, On
spectral polynomials of the Heun equation. II.

\bibitem {pritsker}Igor E. Pritsker. How to find a measure from its potential.
Computational Methods and Function Theory, Volume 8 (2008),No.2, 597-614.
Funktionentheorie, Springer-Verlag, Berlin, 1958. Zbl 0083.29606, MR0096806.

\bibitem {positive}J.-E. Bjork, J. Borcea, R.B\H{o}gvad, Subharmonic
Configurations and Algebraic Cauchy Transforms of Probability Measures.
Notions of Positivity and the Geometry of Polynomials Trends in Mathematics
2011, pp 39-62.

\bibitem {Shapiro}Rikard B\H{o}gvad and Boris Shapiro, On motherbody measures
and algebraic Cauchy transform

\bibitem {Shapiro2}Shapiro, B., Takemura, K. \& Tater, M. Commun. Math. Phys.
(2012) 311: 277. https://doi.org/10.1007/s00220-012-1466-3.

\bibitem {barishnikov}Yuliy Baryshnikov, On stokes Setsin
%TCIMACRO{\TEXTsymbol{\backslash}}%
%BeginExpansion
$\backslash$%
%EndExpansion
New developments in singularity theory (Cambridge,2000)", 65-86, NAT O Sci.
Ser. II Math. Phys. Chem., Vol. 21, Kluwer Acad. Publ., Dordrecht,(2001).

\bibitem {garnet}J.B. Garnett, Analytic capacity and measure, LNM 297,
Springer-Verlag, 1972, 138 pp.
\end{thebibliography}
\end{document}